\documentclass[a4paper,centertags,oneside,11pt]{amsart}
\usepackage[citecolor=red,linkcolor=blue]{hyperref}
\usepackage{mathrsfs}
\usepackage{appendix}
\usepackage{amssymb}
\usepackage{fancyhdr}
\usepackage{charter}
\usepackage{typearea}
\usepackage{pdfsync}
\usepackage{amsmath,amstext,amsthm,amscd}
\usepackage{dsfont}
\usepackage{mathrsfs}
\usepackage[a4paper,top=3cm,bottom=3cm,left=3cm,right=3cm]{geometry}

\usepackage{enumitem}
\setlist[1]{itemsep=5pt}
\newcommand{\comment}[1]{}

\makeatletter
      \def\@setcopyright{}
      \def\serieslogo@{}
      \makeatother

\newcommand{\Complex}{\mathbb C}
\newcommand{\C}{\mathbb C}
\newcommand{\Real}{\mathbb R}

\newcommand{\ddbar}{\overline\partial}

\newcommand{\ol}{\overline}
\newcommand{\Td}{\widetilde}
\newcommand{\norm}[1]{\left\Vert#1\right\Vert}
\newcommand{\abs}[1]{\left\vert#1\right\vert}

\newcommand{\set}[1]{\left\{#1\right\}}
\newcommand{\To}{\rightarrow}

\newtheorem{theorem}{Theorem}[section]
\newtheorem{lemma}[theorem]{Lemma}
\newtheorem{corollary}[theorem]{Corollary}
\newtheorem{proposition}[theorem]{Proposition}
\theoremstyle{definition}
\newtheorem{definition}[theorem]{Definition}

\newtheorem{example}[theorem]{Example}
\newtheorem{remark}[theorem]{Remark}

\numberwithin{equation}{section}
\begin{document}
\title[]{On the stability of equivariant embedding of compact CR manifolds with circle action}
\author[]{Chin-Yu Hsiao}
\address{Institute of Mathematics, Academia Sinica and National Center for Theoretical Sciences, Astronomy-Mathematics Building, No. 1, Sec. 4, Roosevelt Road, Taipei 10617, Taiwan}
\thanks{Chin-Yu Hsiao was partially supported by Taiwan Ministry of Science of Technology project
104-2628-M-001-003-MY2, the Golden-Jade fellowship of Kenda Foundation and Academia Sinica Career Development Award.}
\email{chsiao@math.sinica.edu.tw or chinyu.hsiao@gmail.com}
\author[]{Xiaoshan Li}
\address{School of Mathematics
and Statistics, Wuhan University, Wuhan 430072, Hubei, China}
\thanks{Xiaoshan Li was supported by  National Natural Science Foundation of China (Grant No. 11501422)}
\email{xiaoshanli@whu.edu.cn}

\author[George Marinescu]{George Marinescu}
\address{Universit{\"a}t zu K{\"o}ln,  Mathematisches Institut,
    Weyertal 86-90,   50931 K{\"o}ln, Germany\\
    \& Institute of Mathematics `Simion Stoilow', Romanian Academy,
Bucharest, Romania}
\thanks{}
\email{gmarines@math.uni-koeln.de}

\begin{abstract}
We prove the stability of the equivariant embedding of
compact strictly pseudoconvex CR
manifolds with transversal CR circle action under circle invariant
deformations of the CR structures.
\end{abstract}
\maketitle
\tableofcontents
\section{Introduction and statement of the main results}\label{Introduction}
Let $X$ be be a compact
strictly pseudoconvex CR manifold.
The question of whether or not $X$ admits a CR embedding into
a complex Euclidean space has attracted a lot attention.
This amounts to showing that the manifold has a sufficiently
rich collection of global CR functions.
It was shown by Boutet de Monvel \cite{BdM1:74b}
that the answer is affirmative if the dimension  of $X$ is at least five.
The obstructions to constructing global CR functions lie in
the first Kohn-Rossi cohomology group $H^1_b(X)$, which is
finite dimensional if $\dim X\geq 5$. An essential ingredient in
the embedding theorem \cite{BdM1:74b} is the Hodge theory for this group,
that will play an important role in the present paper, too.

In contrast, if $X$ has dimension three,
$X$ may not be even locally embeddable, see \cite{JT82,Ku82,Ni74}.
Furthermore, there are examples \cite{AS70,Bur:79,G94,Ro65} which show that even
when the CR structure on $X$ is locally embeddable (for example, when it is real analytic),
it can happen that the global CR functions on $X$ fail to separate points of $X$.
It was shown in \cite{BE90} that, in a rather precise sense,
``generic'' perturbations of the standard structure on the three
sphere are nonembeddable.

On the other hand, if a compact three dimensional strictly pseudoconvex CR manifold admits a
transversal CR $S^1$-action, it was shown by Lempert \cite{Le92},
Epstein \cite{Ep92} and recently in \cite{HM14,HHL15} by using the Szeg\H{o} kernel,
that such CR manifolds can always be CR embedded into a complex Euclidean space.

In recent years, much progress has been made in understanding the
embedding question from a deformational point of view,
that is, for CR structures which lie in a small neighborhood
of a fixed embedded structure,
see e.\,g.\ \cite{BlDu91,BE90,Ep92,EH00,HLY06,C15,Le92,Le94,W04}.

There are several distinct notions of stability:

(1) A CR-structure $(X,J)$ is said to be stable provided that the
entire algebra of CR functions deforms continuously under any sufficiently
small embeddable deformation $J'$.

(2) A CR-structure $(X,J)$ is said to be stable for a class of
embeddable deformations $\mathcal{F}$ provided that the entire algebra of CR-functions
deforms continuously under any sufficiently
small deformation $J'\in\mathcal{F}$.

(3) An embedding $F:(X,J)\to\C^N$ is stable for a class $\mathcal{F}$
of embeddable deformations, provided that for each $J'\in\mathcal{F}$
sufficiently close to $J$, there is an embedding
$F':(X,J')\to\C^N$, so that $F'$ is a small perturbation of $F$.

Notion (1) of course implies that, for any given embedding $F:(X,J)\to\C^N$,
there is a nearby embedding $F':(X,J')\to\C^N$, provided that
$(X,J')$ is embeddable and $J'$ is sufficiently close to $J$.
We say that two tensors are close if they are close in the
$C^\infty$ topology on the appropriate space.
For the round 3-sphere the first and second notions, while not explicitly stated,
already appear in Burns and Epstein \cite{BE90}, where it is demonstrated that the entire algebra
of CR functions is stable for the class of ``positive'' deformations, with no requirement
of $S^1$-invariance.
This work was extended by Epstein to positive deformations of circle bundles in \cite{Ep92}.
Lempert \cite{Le92} showed that all small embeddable deformations of the round sphere are,
in fact, positive. In a later paper \cite{Le94} he went on to show that all small embeddable deformations of
CR-structures on the boundaries of strictly pseudoconvex domains in $\C^2$ are stable
in the strongest sense, (1), above.

In the present paper we will only use the notion (3) of stability.
When $X$ is strictly pseudoconvex, of dimension at least five,
Tanaka \cite{T75} proved the stabilty in the sense (3),
provided the dimension of the Kohn-Rossi cohomology $H^{1}_{b}(X)$
is independent of CR structure.
Huang, Luk and Yau \cite{HLY06} studied the stability of embeddings for a
family of strongly pseudoconvex CR manifolds depending in a CR way on
the parameters. The CR dependence on the parameters
is crucial for the study of deformations of complex structures of
isolated singularities. For this topic we refer the readers to
\cite{BJ97, M99, H08, HLY06} and the references therein.

On the other hand, in the case of three dimensional strictly pseudoconvex
CR manifolds, Catlin and Lempert \cite{CL92}
showed that unstable CR embeddings exist.
The CR manifolds
with unstable embeddings arise as unit circle bundles in Hermitian
line bundles over projective manifolds. The instability of CR embeddings is
a consequence of the instability of very ampleness of line bundles.

As mentioned above, the stability of CR embeddings is closely related to the stability of the first
Kohn-Rossi cohomology (see \cite{T75, HLY06}).
Recently, it was shown in \cite{HHL15} that a compact strictly
pseudoconvex CR manifold with a locally free transversal CR $S^1$-action
can be CR embedded into some complex Euclidean space by CR functions lying
in the Fourier components with large positive frequency of the space of CR functions.
Since Fourier components with large frequency of the first Kohn-Rossi
cohomology vanish uniformly under $S^1$-invariant deformations of the
CR structure (see Theorem~\ref{main-160302}), we can expect in analogy
to \cite{T75,HLY06} that the CR
embedding established in \cite{HHL15} should be stable under the $S^1$-invariant deformations.
We will prove this using an additional argument, the stability
of the Szeg\H{o} projector. Similar arguments can  be found in a series of papers by
Epstein \cite{E06, E07, E07a} on relative index,
where the Szeg\H{o} projector also plays a central role.

Let us now formulate our main results. We refer to Section \ref{s-16030203} for some standard
notations and terminology used here.
Let $(X,HX,J)$ be a compact CR manifold
of dimension $2n-1$, $n\geqslant2$, endowed with a locally free $S^1$-action
$S^1\times X\to X$, $(e^{i\theta},x)\mapsto e^{i\theta}  x$ and we let $T$
be the infinitesimal generator of the $S^1$-action. We assume that this $S^1$-action is transversal CR, that is, $T$ preserves the CR structure $T^{1,0}X$,
and $T$ and $T^{1,0}X\oplus\overline{T^{1,0}X}$ generate
the complex tangent bundle to $X$. Let $\overline\partial_b$ be the tangential
Cauchy-Riemann operator on $X$. We denote by
${\rm Ker}(\overline\partial_b)=\{u\in C^\infty(X): \overline\partial_b u=0\}$
the space of smooth CR functions.
For any $m\in\mathbb Z$, we define the $m$-th Fourier component of CR functions
$H^0_{b, m}(X)=\{u\in{\rm Ker}(\overline\partial_b): Tu=imu\}.$
It was shown in \cite{HHL15} that $X$ can be CR embedded into complex Euclidean space by
CR functions
which lie in the Fourier components of CR functions with large positive frequency $m$.
Precisely, for every $m\in\mathbb N$,
there exist integers $\{m_j\}_{j=1}^N$ with $m_j\geq m$, $1\leq j\leq N$, and
CR functions $\{f_j\}_{j=1}^N$ with $f_j\in H^0_{b, m_j}(X)$
such the (equivariant) CR map from $X$ to $\mathbb C^N$
\begin{equation}
\Phi: X\rightarrow \mathbb C^N, ~ x\mapsto (f_1(x), \ldots, f_N(x)),
\end{equation}
is an embedding.
Our goal is to show that such an embedding
is stable under  $S^1$-invariant deformations of the CR structure
(cf.\ Definition~\ref{g-160301b}).
Our main result is the following.
\begin{theorem}\label{mtheorem-160323}
Let $(X, HX, J)$ be a compact connected strictly pseudoconvex CR
manifold with a  locally free transversal CR $S^1$-action.
Let $\{J_t\}_{t\in(-\delta_0, \delta_0)}$
be any $S^1$-invariant deformation of $J$.
Then there is a positive integer $m_0$ such that every CR embedding
$\Phi=(\Phi_1, \ldots, \Phi_N): (X,HX,J)\rightarrow\mathbb C^N$ with
$\Phi_j\in H^0_{b, m_j}(X), m_j>m_0, j=1,\ldots,N$,
is stable with respect to the deformation
$\{J_t\}_{t\in(-\delta_0,\delta_0)}$, that is,
for $|t|$ small enough there exists
a $S^1$-equivariant CR embedding $f_t$ of the structure $J_t$ such that
$f_t$ converges to $f$ as $t\to0$ in the $C^m$ topology for any non-negative $m\in\mathbb Z$.
\end{theorem}


This paper is organized as follows. In Section \ref{prem160323},
we set up notation and terminology.  Section \ref{section-160322}
is devoted to study the $S^1$-invariant deformation of CR structure.
Furthermore, will prove the simultaneous vanishing theorem of Fourier
component of Kohn-Rossi cohomology. In Section \ref{stability-160323},
we will be concerned with the stability of the Szeg\H{o}
kernel of Fourier components of Kohn-Rossi cohomology.
Using the stability of Szeg\H{o} kernel,
we will prove Theorem \ref{mtheorem-160323} in Section \ref{s_stab}.

\section{Preliminaries}\label{prem160323}

\subsection{Set up and terminology}\label{s-16030203}
Let $(X, T^{1,0}X)$ be a compact CR manifold of dimension $2n-1, n\geq 2$,
where $T^{1,0}X$ is a CR structure of $X$,
that is, $T^{1,0}X$ is a subbundle of the complexified tangent
bundle $\mathbb{C}TX$ of rank $n-1$ satisfying
$T^{1,0}X\cap T^{0,1}X=\{0\},$ where $T^{0,1}X=\overline{T^{1,0}X}$
and $[\mathcal V,\mathcal V]\subset\mathcal V,$ where $\mathcal V=C^\infty(X, T^{1,0}X).$
There is a unique subbundle $HX$ of $TX$ such that
$\mathbb CHX=T^{1, 0}X\bigoplus T^{0, 1}X$, i.e.,
$HX$ is the real part of $T^{1, 0}X\bigoplus T^{0, 1}X.$
Furthermore, there exists a homomorphism $J: HX\rightarrow HX$ such that $J^2=-{\rm id}$,
where $\rm id$ denotes the identity ${\rm id}: \mathbb CHX\rightarrow\mathbb CHX$.
By complex linear extension of $J$ to $\mathbb CTX$, the $i$-eigenspace of $J$
is given by $T^{1, 0}X=\{V\in\mathbb CHX: JV=iV\}$. We shall also write $(X, HX, J)$
to denote a compact CR manifold.
Let $E$ be a smooth vector bundle over $X$.
We use $\Gamma(E)$ to denote the space of $C^\infty$-smooth sections of $E$ on $X$.

Let $(X, HX, J)$ be a compact CR manifold. Let $\Omega\subset\Real$ be an open neighborhood of $0$.
We say that $\set{J_t}_{t\in\Omega}$ is a deformation of $J$ if

(I) For each $t\in\Omega$, there is an endomorphism
$J_t: HX\rightarrow HX$ with $J_t^2=-{\rm id}$ and the $i$ eigenspace
$T_t^{1, 0}X=\{U\in \mathbb CHX: J_t U=iU\}$  is a CR structure on $X$.

(II) $J_0=J$.

(III) $J_t$  depends smoothly
on $t$, that is, for every $U\in HX$ and $V^*\in T^*X$ we have
$\langle\,J_tU\,,\,V^*\,\rangle\in C^\infty(\Omega)$.

From now on, we assume that $(X, HX,J)$ admits a $S^1$-action:
$S^1\times X\rightarrow X, (e^{i\theta}, x)\mapsto e^{i\theta}\circ x$.
Here, we use $e^{i\theta}$ to denote the $S^1$-action.
Let $T\in C^\infty(X, TX)$ denote the global real vector field induced by the $S^1$- action given as follows
\begin{equation}\label{e-160301}
(Tu)(x)=\frac{\partial}{\partial\theta}\left(u(e^{i\theta}\circ x)\right)\Big|_{\theta=0}\,, \:\:u\in C^\infty(X).
\end{equation}
We say that the $S^1$-action $e^{i\theta} (0\leq\theta<2\pi$) is CR if
\begin{equation}\label{b-160228}
[T, \Gamma(T^{1, 0}X)]\subset \Gamma(T^{1, 0}X),
\end{equation}
where $[\cdot\,,\cdot]$ denotes the Lie bracket between the smooth vector fields on $X$.
Furthermore, we say that the $S^1$- action is transversal if for each $x\in X$,
\begin{equation}
\mathbb CT_xX=\mathbb CT(x)\oplus T_x^{1,0}(X)\oplus T_x^{0,1}X.
\end{equation}
From now on, we assume that the $S^1$-action on $(X, HX,J)$ is transversal and CR.
Let $\set{J_t}_{t\in\Omega}$ be  a deformation of $J$, where $\Omega\subset\Real$
is an open neighborhood $0\in\Omega$.
As before, put $T_t^{1, 0}X=\{U\in \mathbb CHX: J_t U=iU\}$. We need
\begin{definition}\label{g-160301b}
With the notations above, we say that $\set{J_t}_{t\in\Omega}$ are $S^1$-invariant deformations of $J$ if $[T, \Gamma(T_t^{1, 0}X)]\subset\Gamma(T_t^{1, 0}X)$ for every $t\in\Omega$.
\end{definition}
\begin{definition}\label{d-gue160409}
Let $f:(X,HX,J)\To\mathbb C^k$ be a CR embedding and let
$\set{J_t}_{t\in\Omega}$ be $S^1$-invariant deformations of $J$,
where $\Omega\subset\Real$ is an open neighborhood of $0$.
We say that $f$ is stable with respect to $\set{J_t}_{t\in\Omega}$
if there is a $\delta>0$ with
$[-\delta,\delta]\subset\Omega$ such that for every $t\in (-\delta,\delta)$,
we can find a CR embedding $f_t: (X,HX,J_t)\To\mathbb C^k$ and
$\lim\limits_{t\To0}\norm{f_t-f}_{C^m(X, \mathbb C^k)}=0$,
for every $m\in\mathbb N_0:=\mathbb N\cup\set{0}$.
\end{definition}
\begin{lemma}\label{l-160302}
With the notations used above,  we have $L_T J=0$ on $HX$, where $L_T$
denotes the Lie derivative along the direction $T$.
\end{lemma}

\begin{proof}
For any $U\in \Gamma(T^{1, 0}X)$, $L_T J(U)=L_T(JU)-JL_TU=
\sqrt{-1}L_TU-\sqrt{-1}L_TU=0.$
Here, we have used the fact that the $S^1$-action is CR,
that is, $L_TU\in\Gamma(T^{1, 0}X)$ for any $U\in\Gamma(T^{1, 0}X).$
For any $V\in\Gamma(T^{0, 1}X)$, we have $L_TJ(V)=\overline{L_TJ(\overline V)}=0.$
Since $HX$ is the real part of $T^{1, 0}X\bigoplus T^{0, 1}X$, the Lemma follows.
\end{proof}

Since $[\Gamma(T^{1, 0}X), \Gamma(T^{1, 0}X)]\subset \Gamma(T^{1, 0}X)$,
we have $[JU, JV]-[U, V]\in C^{\infty}(X, HX)$ for all $U, V\in C^{\infty}(X, HX)$.
Let $\omega_0$ be the global real $1$-form dual to $T$, that is,
\begin{equation}\label{f-160301}
\langle\omega_0, T\rangle=1, \langle \omega_0, HX\rangle=0.
\end{equation}
Then for each $x\in X$, we define a quadratic form on $HX$ by
\begin{equation}\label{e-160404}
\mathcal L_x(U, V)=-d\omega_0(JU, V), \forall~ U, V\in H_x X.
\end{equation}
The quadratic form is called the Levi form at $x$. We extend $\mathcal L$ to
$\mathbb CHX$ by complex linear extension. Then for $U, V\in T_x^{1, 0}X$,
\begin{equation}
\mathcal L_x(U, \overline V)=-d\omega_0(JU, \overline V)=-id\omega_0(U, \overline V).
\end{equation}
\begin{definition}
We say $T^{1, 0}X$ is a strictly pseudoconvex structure and $X$
is a strictly pseudoconvex CR manifold if the Levi form $\mathcal L_x$
is a positive definite quadratic form on $H_x X$ for each $x\in X$.
\end{definition}

In the following, we always assume that  $X$ is a compact connected strictly
pseudoconvex CR manifold with a transversal CR $S^1$-action.
It should be noted that a strictly pseudoconvex CR manifold is
always a contact manifold. From (\ref{f-160301}), we see that
$\omega_0$ is a contact form, $HX$ is the contact plane and $T$
is the Reeb vector field. Using (\ref{e-160404}) we may extend the
Levi form to a Riemannian metric $g$ on $TX$, which will play a crucial role in the sequel.

\begin{definition}
Let $X$ be a compact strictly pseudoconvex CR manifold with a
transversal CR $S^1$-action. Let $g$ be the Riemannian metric given by
\begin{equation}\label{a-160227}
g(U, V)=\mathcal L_x(U, V),\quad g(U, T)=g(T, U)=0, \quad g(T, T)=1,
\end{equation}
for any $U, V\in H_xX, x\in X$. This is called the Webster metric on $X$.
\end{definition}

The volume form associated with the Webster metric is denoted by $dv_X$ and by direct calculation
\begin{equation}\label{volume}
dv_X=\omega_0\wedge\frac{(d\omega_0)^{n-1}}{(n-1)!}.
\end{equation}
The volume form $dv_X$ associated with the Webster metric depends only on the contact form $\omega_0$.


For $U, V\in T^{1, 0}_xX$, we can check that
$\mathcal L_x(U, V)=\langle\,\omega_0(x),[J\mathscr U,\mathscr V](x)\,\rangle=0$, where $\mathscr U, \mathscr V\in\Gamma(T^{1,0}X)$ with $\mathscr U(x)=U$, $\mathscr V(x)=V$.
Thus, $\mathcal L_x(U, \overline V)=-i d\omega_0(U, \overline V)$ is a positive definite Hermitian quadratic form on $T^{1, 0}X$. We extend the Webster metric $g$ to $\mathbb CTX$ by complex linear extension. The Webster metric $g$ on $X$ induces a Hermitian metric $\langle\,\cdot\,|\,\cdot\,\rangle_g$ on $\Complex TX$:
\begin{equation}\label{e-gue160410}
\langle\,U\,|\,V\,\rangle_g:=g(U,\ol V),\ \ U, V\in\Complex TX.
\end{equation}
It is easy to check that the Webster metric is $J$-invariant on $HX$,
so we have the pointwise orthogonal  decomposition
\begin{equation}\label{e160407}
\mathbb CT_xX=\mathbb CT(x)\oplus T_x^{1,0}(X)\oplus T_x^{0,1}X.
\end{equation}
We call $\langle\,\cdot\,|\,\cdot\,\rangle_g$ the Webster Hermitian metric.

Denote by $T^{\ast 1,0}X$ and $T^{\ast0,1}X$ the dual bundles of
$T^{1,0}X$ and $T^{0,1}X$, respectively. Define the vector bundle of $(0,q)$-forms by
$\Lambda^qT^{\ast0,1}X$. Let $D\subset X$ be an open subset. Then $\Omega^{0,q}(D)$
denotes the space of smooth sections of $\Lambda^qT^{\ast0,1}X$ over $D$.

Fix $\theta_0\in [0, 2\pi)$. Let
$$d e^{i\theta_0}: \mathbb CT_x X\rightarrow \mathbb CT_{e^{i\theta_0}\circ x}X$$
denote the differential map of $e^{i\theta_0}: X\rightarrow X$.
By the property of transversal CR $S^1$-action, we can check that
\begin{equation}\label{a}
\begin{split}
de^{i\theta_0}:T_x^{1,0}X\rightarrow T^{1,0}_{e^{i\theta_0}\circ x}X,\\
de^{i\theta_0}:T_x^{0,1}X\rightarrow T^{0,1}_{e^{i\theta_0}\circ x}X,\\
de^{i\theta_0}(T(x))=T(e^{i\theta_0}\circ x).
\end{split}
\end{equation}
Let $(e^{i\theta_0})^\ast: \Lambda^q(\mathbb CT^\ast X)\rightarrow
\Lambda^q(\mathbb CT^\ast X)$ be the pull back of $e^{i\theta_0}, q=0,1,\ldots, n-1$.
From \eqref{a}, we can check that for every $q=0, 1,\ldots, n-1$
\begin{equation}\label{j1}
(e^{i\theta_0})^\ast: \Lambda^qT^{\ast0,1}_{e^{i\theta_0}\circ x}X\rightarrow \Lambda^qT_x^{\ast0,1}X.
\end{equation}

For $u\in\Omega^{0,q}(X)$ we define $Tu$ as follows:
\begin{equation}\label{b}
(Tu)(X_1,\ldots, X_q):=\frac{\partial}{\partial\theta}
\left((e^{i\theta})^\ast u(X_1,\ldots, X_q)\right)\Big|_{\theta=0}\,,
\quad X_1,\ldots, X_q\in T_x^{1,0}X.
\end{equation}
From (\ref{j1}) and (\ref{b}), we have  $Tu\in\Omega^{0, q}(X)$ for all
$u\in\Omega^{0, q}(X)$. From the definition of $Tu$ it is easy to check that
$Tu=L_Tu$ for $u\in\Omega^{0, q}(X)$, where $L_Tu$ is the Lie derivative
of $u$ along the direction $T$.

Let $\overline\partial_b:\Omega^{0,q}(X)\rightarrow\Omega^{0,q+1}(X)$
be the tangential Cauchy-Riemann operator. It is straightforward from (\ref{a}) and (\ref{b}) to see that
\begin{equation}\label{c}
T\overline\partial_b=\overline\partial_bT~\text{on}~\Omega^{0,q}(X).
\end{equation}
For every $m\in\mathbb Z$, put $\Omega^{0,q}_m(X):=\{u\in\Omega^{0,q}(X): Tu=imu\}$.
From (\ref{c}) we have the $\ddbar_b$-complex for every $m\in\mathbb Z$:
\begin{equation}\label{e-gue140903VI}
\ddbar_b:\ldots\To\Omega^{0,q-1}_m(X)\To\Omega^{0,q}_m(X)\To\Omega^{0,q+1}_m(X)\To\ldots.
\end{equation}
For every $m\in\mathbb Z$, the $m$-th Fourier component of Kohn-Rossi cohomology is defined as follows
\begin{equation}\label{a8}
H^{q}_{b,m}(X):=\frac{{\rm Ker\,}\ddbar_{b}:\Omega^{0,q}_m(X)\To\Omega^{0,q+1}_m(X)}{\operatorname{Im}\ddbar_{b}:\Omega^{0,q-1}_m(X)\To\Omega^{0,q}_m(X)}\,\cdot
\end{equation}

\begin{definition}
We say that a function $u\in C^\infty(X)$ is a Cauchy-Riemann (CR for short) function
if $\overline\partial_bu=0$, or in the other words,
$\overline Zu=0$ for all $Z\in\Gamma(T^{1, 0}X)$.
\end{definition}

For $m\in\mathbb Z$, when $q=0$, $H^0_{b, m}(X)$ is a subspace of the space of CR functions which lie in the $im$ eigenspace of $T$  and we call $H^0_{b, m}(X)$ the $m$-th Fourier component of the space of CR functions.

\subsection{Canonical local coordinates}\label{q}

In this work, we need the following result due to Baouendi-Rothschild-Treves.

\begin{theorem}\label{e-can-160301}\cite[Proposition I.2]{BRT85}
Let $X$ be a compact CR manifold of ${\rm dim}X=2n-1, n\geq2$ with a transversal CR $S^1$-action. For  $x_0\in X$, there exist local coordinates $(x_1,\ldots,x_{2n-1})=(z,\theta)=(z_1,\ldots,z_{n-1},\theta), z_j=x_{2j-1}+ix_{2j},j=1,\ldots,n-1, x_{2n-1}=\theta$, defined in a small neighborhood $D=\{(z, \theta): |z|<\varepsilon, |\theta|<\delta\}$ centered at $x_0$ such that
\begin{equation}\label{e-can}
\begin{split}
&T=\frac{\partial}{\partial\theta}\\
&Z_j=\frac{\partial}{\partial z_j}+i\frac{\partial\varphi(z)}{\partial z_j}\frac{\partial}{\partial\theta},j=1,\ldots,n-1
\end{split}
\end{equation}
where $\{Z_j(x)\}_{j=1}^{n-1}$ form a basis of $T_x^{1,0}X$ for each $x\in D$, and $\varphi(z)\in C^\infty(D,\mathbb R)$ is independent of $\theta$.
\end{theorem}
We call $D$ a canonical local patch, $x=(z, \theta)$ canonical local coordinates on $D$ and $\{Z_j\}_{j=1}^{n-1}$ a canonical frame of $T^{1, 0}X$ over $D$. On $D$, the contact form is given by $$\omega_0=d\theta-i\sum_{j=1}^{n-1}\frac{\partial\varphi(z)}{\partial z_j}dz_j+i\sum_{j=1}^{n-1}\frac{\partial\varphi(z)}{\partial\overline z_j}d\overline z_j$$
and the Levi form on $T^{1, 0}X$ can be expressed as
\begin{equation}
\mathcal L_x=-id\omega_0=2\sum_{k, j=1}^{n-1}\frac{\partial^2\varphi(z)}{\partial z_k\partial\overline z_j}dz_k\wedge d\overline z_j.
\end{equation}

For $x\in D$, $\theta\in [0, 2\pi)$ with $e^{i\theta}\circ x\in D$, it is straightforward to see that $de^{i\theta}(Z_j(x))=Z_j(e^{i\theta}\circ x)$ for $ 1\leq j\leq n-1.$
\subsection{Hermitian CR geometry}
\begin{definition}\cite[{Definition 1.18}]{H14}
Let $D$ be an open set and let $V\in C^\infty(D, \mathbb CTX)$ be a vector field on $D$. We say that $V$ is rigid if
\begin{equation}
de^{i\theta}(V(x))=V(e^{i\theta}\circ x)
\end{equation}
for any $x, \theta\in[0,2\pi)$ satisfying $x\in D, e^{i\theta}\circ x\in D.$
\end{definition}
The canonical frame $\{Z_j\}_{j=1}^{n-1}$ defined  in (\ref{e-can}) are rigid vector fields on the canonical local patch. Let $D$ be an open subset of $X$ and $U$ be a rigid vector field on $D$. Then for any $\theta_0\in[0, 2\pi)$, $de^{i\theta_0}(U)$ is still a rigid vector field on $e^{i\theta_0} D:=\{e^{i\theta_0}\circ x: x\in D\}.$
\begin{definition}\cite[Definition 1.19]{H14}
Let $\langle\cdot\,|\,\cdot\rangle$ be a Hermitian metric on $\mathbb CTX$.
We say that $\langle\cdot\,|\,\cdot\rangle$ is rigid if for rigid vector fields $V, W$ on $\Omega$, where $\Omega$ is any open set, we have
\begin{equation}
\langle V(x)|W(x)\rangle=\langle (de^{i\theta}V)(e^{i\theta}\circ x)|(de^{i\theta}W)(e^{i\theta}\circ x)\rangle, \forall x\in \Omega, \theta\in[0,2\pi).
\end{equation}
\end{definition}
\begin{lemma}
The Webster Hermitian metric $\langle\,\cdot\,|\,\cdot\,\rangle_g$ defined in (\ref{e-gue160410}) is a rigid Hermitian metric on $\mathbb CTX$.
\end{lemma}
\begin{proof}
\underline{}Let $\Omega$ be an open subset of $X$ and $U, V\in T^{1, 0}X$ be rigid vector fields on $\Omega$. For any $x_0\in \Omega$, choose  canonical  coordinates $x=(z, \theta)$ centered at $x_0$ and a canonical local patch $D=\{(z, \theta): |z|<\varepsilon, |\theta|<\delta\}$ with $\overline D\subset \Omega.$ Let $\{Z_j\}_{j=1}^{n-1}$ be a canonical frame over $D$. Then on $D$, $U=\sum_{j=1}^{n-1}a_j(z, \theta)Z_j$ and  $V=\sum_{j=1}^{n-1}b_j(z, \theta) Z_j.$ Since $U, V$ are rigid vector fields we have that on $D$, $\frac{\partial}{\partial\theta}a_j(z, \theta)=\frac{\partial}{\partial\theta}b_j(z, \theta)=0 $ for $1\leq j\leq n-1.$ Then for $|\theta|<\delta$,
\begin{equation}\label{04-30-2017-3}
\langle de^{i\theta}U(x_0)\,|\, de^{i\theta} V(x_0)\rangle_g=\sum_{j, k=1}^{n-1}a_j(0, 0)\overline b_k(0, 0)\langle de^{i\theta}Z_j(x_0)\,|\, de^{i\theta} Z_k(x_0)\rangle_g.
\end{equation}
Substituting $de^{i\theta}Z_j(x_0)=\frac{\partial}{\partial z_j}+i\frac{\partial\varphi}{\partial z_j}(0)\frac{\partial}{\partial\theta}|_{(0, \theta)}$ to (\ref{04-30-2017-3}) we have
\begin{equation}\label{ind1-30-04-2017}
\langle de^{i\theta}U(x_0)\,|\, de^{i\theta} V(x_0)\rangle_g
=\langle U(x_0)\,|\, V(x_0)\rangle_g, \forall~ |\theta|<\delta.
\end{equation}
Now, we claim that the above equality is also true for all $\theta\in[0, 2\pi)$.
Let $0<\delta_1<2\pi$ be any number
such that \begin{equation}\label{04-30-2017-4}
\langle de^{i\theta}U(x_0)\,|\, de^{i\theta} V(x_0)\rangle_g=\langle U(x_0)\,|\, V(x_0)\rangle_g, \forall~ 0\leq\theta<\delta_1.
\end{equation}
First, we show that
\begin{equation}\label{04-30-2017-2}
\langle de^{i\delta_1}U(x_0)\,|\, de^{i\delta_1} V(x_0)\rangle_g= \langle U(x_0)\,|\, V(x_0)\rangle_g.
\end{equation}
Set $U_1=de^{i\delta_1}U$, $V_1=de^{i\delta_1}V$ and $y_0=e^{i\delta_1}\circ x_0$.
Since $U_1, V_1$ are still rigid vector fields on $e^{i\delta_1}\Omega$, then by the same argument in the proof of (\ref{ind1-30-04-2017}), there exist $\sigma>0$ such that
\begin{equation}\label{04-30-2017-5}
\langle de^{i\theta}U_1(y_0)\,|\, de^{i\theta} V_1(y_0)\rangle_g= \langle U_1(y_0)\,|\,V_1(y_0)\rangle_g, ~\forall |\theta|<\sigma.
\end{equation}
Thus, by (\ref{04-30-2017-4}) and (\ref{04-30-2017-5}) we have
\begin{equation}\label{04-30-2017-6}
\langle de^{i\delta_1}U(x_0)\,|\,de^{i\delta_1} V(x_0)\rangle_g= \langle de^{i(\delta_1-\frac{\sigma}{2})}U(x_0) \,|\, de^{i(\delta_1-\frac{\sigma}{2})}V(x_0)\rangle_g=\langle U(x_0)\,|\,V(x_0)\rangle_g.
\end{equation}
Then we get the conclusion of (\ref{04-30-2017-2}). On the other hand, by (\ref{04-30-2017-5}) and (\ref{04-30-2017-6}) we have
\begin{equation}\label{04-30-2017-7}
\langle U(x_0)\,|\,V(x_0)\rangle_g=\langle de^{i(\delta_1+\varepsilon)}U(x_0)\,|\, de^{i(\delta_1+\varepsilon)}V(x_0)\rangle, ~\forall~\varepsilon\in (0, \sigma).
\end{equation}
Thus, from (\ref{04-30-2017-2}) and (\ref{04-30-2017-7}) we have
\begin{equation}\label{04-30-2017-8}
\langle de^{i\theta}U(x_0)\,|\, de^{i\theta}V(x_0)\rangle=\langle U(x_0)\,|\, V(x_0)\rangle_g, ~\forall~ 0\leq\theta<\delta_1+\sigma.
\end{equation}
Combining (\ref{04-30-2017-4}) and (\ref{04-30-2017-8}) we get the conclusion of the claim and the lemma follows.
\end{proof}
For the existence of rigid Hermitian metric on general CR manifold with $S^1$-action, we refer the readers to \cite[Theorem 9.2]{H14}.

From now on, we will fix the Webster Hermitian metric as a rigid Hermitian metric on $\mathbb CTX$. For convenience, we use the notation $\langle\cdot\,|\,\cdot\rangle$ to denote $\langle\cdot\,|\,\cdot\rangle_g$. The rigid Hermitian  metric $\langle\cdot\,|\,\cdot\rangle$ on
$\mathbb CTX$ induces by duality a Hermitian metric on $\mathbb CT^\ast X$ and also on the bundles of $(0,q)$-forms $\Lambda^qT^{\ast 0,1}X,$ $ q=0,1\ldots,n-1.$ We shall also denote all these induced
metrics by $\langle\cdot\,|\,\cdot\rangle$. From (\ref{e160407}) we have the pointwise orthogonal decomposition:
\begin{equation}
\mathbb CT^{\ast}X=T^{\ast1,0}X\oplus T^{\ast0,1}X\oplus\{\lambda\omega_0:\lambda\in\mathbb C\}.
\end{equation}
For every $v\in \Lambda^qT^{\ast0,1}X$, we write
$|v|^2:=\langle v|v\rangle$.
Let $(\,\cdot\,|\,\cdot\,)$ be the $L^2$ inner product on $\Omega^{0,q}(X)$ induced by $\langle\,\cdot\,|\,\cdot\,\rangle$ and let $\norm{\cdot}$ denote the corresponding norm. Then for all $u, v\in\Omega^{0,q}(X)$
\begin{equation}
(u|v)=\int_X\langle u| v\rangle dv_X,
\end{equation}
where $dv_X$ given in (\ref{volume}) is the volume form on $X$ induced by the rigid Hermitian metric.
As before, for $m\in\mathbb Z$, we denote by
\begin{equation}
\Omega_m^{0,q}(X)=\{u\in\Omega^{0,q}(X): Tu=imu\}
\end{equation}
the $im$ eigenspace of $T$.  Let $L^2_{(0, q), m}(X)$ be the completion of $\Omega^{0, q}_m(X)$ under the $L^2$ inner product.

Let $\overline\partial_{b}^\ast: \Omega^{0,q+1}(X)\rightarrow\Omega^{0,q}(X)$ be the formal adjoint of $\overline\partial_b$ with respect to $(\cdot\,|\,\cdot)$. Since the Hermitian metrics $\langle\cdot\,|\,\cdot\rangle$ are rigid, we can check that
\begin{equation}\label{d}
T\overline\partial_{b}^\ast=\overline\partial_{b}^\ast T~\text{on}~\Omega^{0,q}(X), \forall q=1,\ldots,n-1
\end{equation}
and from (\ref{d}) we have
\begin{equation}\label{e}
\overline\partial_{b}^\ast:\Omega_m^{0,q+1}(X)\rightarrow\Omega_m^{0,q}(X),\forall m\in\mathbb Z.
\end{equation}
Put
$$\Box_{b}^{(q)}:=\overline\partial_b\overline\partial_{b}^\ast
+\overline\partial_{b}^\ast\overline\partial_b:\Omega^{0,q}(X)\rightarrow\Omega^{0,q}(X).$$
Combining (\ref{c}), (\ref{d}) and (\ref{e}), we have
\begin{equation}\label{a9}
T\Box_{b}^{(q)}=\Box^{(q)}_{b}T~\text{on}~\Omega^{0,q}(X), \forall q=0,1,\ldots,n-1.
\end{equation}
A direct consequence of (\ref{a9}) is
\begin{equation}
\Box_{b}^{(q)}:\Omega_m^{0,q}(X)\rightarrow\Omega_m^{0,q}(X),\forall m\in\mathbb Z.
\end{equation}
We will write $\Box^{(q)}_{b, m}$ to denote the restriction of $\Box^{(q)}_{b}$ on $\Omega_m^{0,q}(X)$. For every $m\in\mathbb Z$, we extend $\Box^{(q)}_{b, m}$ to $L^2_{(0,q),m}(X)$ by
\begin{equation}\label{j2}
\Box^{(q)}_{b, m}:{\rm Dom}(\Box^{(q)}_{b,m})\subset L^2_{(0,q),m}(X)\rightarrow L^2_{(0,q),m}(X),
\end{equation}
where ${\rm Dom}(\Box^{(q)}_{b,m})=\{u\in L^2_{(0,q),m}(X): \Box^{(q)}_{b, m}u\in L^2_{(0,q),m}(X)~~\text{in the sense of distribution}\}$.

The following result follows from Kohn's $L^2$-estimate (see Theorem 8.4.2 in \cite{CS01}).

\begin{theorem}\label{f}
For every $s\in\mathbb{N}_0$, there exists a constant $C_{s}>0$ such that
\begin{equation}
\|u\|_{s+1}\leq C_{s}\left(\|\Box^{(q)}_{b}u\|_{s}+\|Tu\|_s+\|u\|_s\right),\forall u\in\Omega^{0,q}(X)
\end{equation}
where $\|\cdot\|_s$ denotes the standard Sobolev norm of order $s$ on $X$.
\end{theorem}

From Theorem \ref{f}, we deduce that

\begin{theorem}\label{g}
For $m\in\mathbb Z$ and for every $s\in\mathbb N_0$, there is a constant $C_{s, m}>0$ such that
\begin{equation}\label{xx}
\|u\|_{s+1}\leq C_{s, m}\left(\|\Box^{(q)}_{b, m}u\|_s+\|u\|_s\right), \forall u\in\Omega^{0,q}_m(X).
\end{equation}
\end{theorem}

According to Theorem~\ref{g} and a standard argument in functional analysis, we deduce the following Hodge theory for $\Box^{(q)}_{b, m}$ (see Section 3 in~\cite{CHT15}).

\begin{theorem}\label{gI}
Let $q\in\{0,1,\ldots,n-1\}$, $m\in\mathbb Z$.
$\Box^{(q)}_{b, m}:\mathrm{Dom}(\Box^{(q)}_{b, m})\subset L^2_{(0,q), m}(X)\rightarrow L^2_{(0,q), m}(X)$ is a self-adjoint operator. Set
\begin{equation}\label{h}
\mathcal{H}^q_{b, m}(X)=\left\{u\in\mathrm{Dom}(\Box^{(q)}_{b, m}):\Box^{(q)}_{b, m}u=0\right\}.
\end{equation}
Then $\mathcal H^q_{b, m}(X)$ is a finite dimensional space with
$\mathcal{H}^q_{b, m}(X)\subset\Omega^{0,q}_m(X)$  and the map
\begin{equation}\label{i}
\mathcal{H}^q_{b, m}(X)\cong H^q_{b,m}(X),\quad \alpha\mapsto[\alpha],
\end{equation}
is an isomorphism, where $[\alpha]$ is the cohomology
class of $\alpha$ in $H^q_{b,m}(X)$.
In particular,
\begin{equation}\label{j3}
\dim H^q_{b,m}(X)<\infty, \forall ~m\in\mathbb Z, \forall ~0\leq q\leq n-1.
\end{equation}
\end{theorem}
We call $\mathcal H^q_{b, m}(X)$ the harmonic space with respect to $\Box^{(q)}_{b, m}$.

\subsection{Tanaka-Webster connection}Let $(X, HX, J)$ be a compact strictly pseudoconvex CR manifold with a transversal CR $S^1$-action. Let $T$ be the globally real vector field induced  by the $S^1$-action and $\omega_0$ be its dual form. Then it is easy to check that $\omega_0$ is a contact form with $HX$ as the contact structure and  $T$, $\omega_0$ satisfy
\begin{equation}
\langle\omega_0,T\rangle=1,~\langle\omega_0, HX\rangle=0,~ T\rfloor d\omega_0=0.
\end{equation}
In this section, with the notations defined above, we will review the Tanaka-Webster connection \cite{T75, W78} and the notions are mainly from \cite{T75, DT06}.

\begin{proposition}[Proposition 3.1 in \cite{T75} ]\label{p-160327}
There is a unique linear connection (Tanaka-Webster connection) denoted by $\nabla: \Gamma(TX)\rightarrow \Gamma(T^\ast X\otimes TX)$ satisfying the following conditions:
\begin{enumerate}
  \item The contact structure $HX$ is parallel, i.e., $\nabla_U \Gamma(HX)\subset \Gamma(HX)$ for $U\in\Gamma(TX).$
  \item The tensor fields $T, J, d\omega_0$ are all parallel, i.e., $\nabla T=0, \nabla J=0, \nabla d\omega_0=0.$
  \item The torsion $\tau$ of $\nabla$ satisfies:
  ${\tau}(U, V)=d\omega_0(U, V)T$, ${\tau}(T, JU)=-J{\tau}(T, U)$,
  $U, V\in C^\infty(X, HX).$
\end{enumerate}
\end{proposition}

Recall that $\nabla J\in\Gamma(T^\ast X\otimes\mathscr L(HX,HX))$,
$\nabla d\omega_0\in\Gamma(T^\ast X\otimes\Lambda^2(\Complex T^\ast X))$
are defined by $(\nabla_UJ)W=\nabla_U(JW)-J\nabla_UW$ and
$\nabla_Ud\omega_0(W, V)=Ud\omega_0(W, V)-d\omega_0(\nabla_UW, V)-
d\omega_0(W, \nabla_UV)$ for $U\in\Gamma(TX), W, V\in\Gamma(HX)$.
Similarly, for any $u\in\Omega^{0,q}(X)$, we can define
$\nabla u\in\Gamma(T^\ast X\otimes\Lambda^q(\Complex T^\ast X))$
in the standard way. By (1) and $\nabla J=0$ in (2), we have
$\nabla_U \Gamma(T^{1, 0}X)\subset \Gamma(T^{1, 0}X)$ and
$\nabla_U \Gamma(T^{0, 1}X)\subset \Gamma(T^{0, 1}X)$ for
$U\in \Gamma(TX).$ Moreover, $\nabla J=0$ and $\nabla d\omega_0=0$
imply that the Tanaka-Webster connection is compatible with the Webster metric.
By definition, the torsion of $\nabla$ is given by
$\tau (W, U)=\nabla_WU-\nabla_UW-[W, U]$ for $U, V\in\Gamma(TX)$ and
$\tau(T, U)$ for $U\in\Gamma(HX)$ is called pseudohermitian torsion.

The existence of an $S^1$-action on $X$ is not necessary
in the definition of Tanaka-Webster connection.
But if $X$ admits a transversal CR $S^1$-action, by (2) in
Proposition \ref{p-160327} and Lemma \ref{d-160228}, we have
$\tau(T, U)=0$ for $U\in\Gamma(HX)$.
The vanishing of pseudohermitian torsion  admits an important
geometric interpretation. Webster \cite{W78} proved that for a strictly
pseudoconvex CR manifold, the pseudohermitian torsion vanishes
if and only if the $1$-parameter group of transformations of
$X$ induced by $T$  consists of CR automorphisms.

\begin{lemma}[Lemma 3.2 in \cite{T75}]\label{l-160228}
Let $U\in\Gamma(T^{1, 0}X)$. Then $\nabla_{T}U=L_TU+J_T U$, where $L_T$ denotes the Lie derivation and $J_T$ is given by $J_T=-\frac12 J\circ L_TJ.$
\end{lemma}
Since the $S^1$-action on $X$ is CR, by (\ref{b-160228}), the CR structure on $X$ is invariant with respect to the $S^1$-action, that is,  $L_T J=0.$ By Lemma \ref{l-160228} and since $L_TJ=0$, we have
\begin{lemma}\label{d-160228}
Let $X$ be a compact strictly pseudoconvex CR manifold with a transversal CR $S^1$-action. Let $\nabla$ be the Tanaka-Webster connection on $TX$. Then we have
\begin{equation}
J_TU=0~\text{and}~\nabla_TU=L_TU~\text{for}~U\in\Gamma(T^{1, 0}X),
\end{equation}
where $T$ denotes the induced vector field by the $S^1$-action and $J$ is the CR structure tensor on $X$.
\end{lemma}
Since $\nabla_Tu\in\Omega^{0, q}(X)$ for $u\in\Omega^{0, q}(X)$, then for any smooth sections $U_1, \ldots, U_q\in\Gamma(T^{0, 1}X)$ we have
\begin{equation*}
\begin{split}
(\nabla_Tu)(\overline U_1, \ldots, \overline U_q)&=T(u(\overline U_1, \ldots, \overline U_q))-\sum_{j=1}^qu((\overline U_1, \ldots, \nabla_T\overline U_j, \ldots, \overline U_q))\\
&=T(u(\overline U_1, \ldots, \overline U_q))-\sum_{j=1}^qu((\overline U_1, \ldots, L_T\overline U_j, \ldots, \overline U_q))\\
&=(L_Tu)(\overline U_1, \ldots, \overline U_q).
\end{split}
\end{equation*}
Thus, we  have $\nabla_Tu=L_Tu$ for $u\in\Omega^{0, q}(X)$.

Let $R$ be the curvature of Tanaka-Webster connection. Let $e_1, \ldots, e_{n-1}$ be any orthonormal basis of $T^{1, 0}X$ with respect to the fixed rigid Hermitian metric $\langle\cdot\,|\,\cdot\rangle$, that is, $\langle e_i| e_j\rangle=\delta_{ij}.$ Then the Ricci curvature operator $R_{\ast}$ is defined by (page 34 in \cite{T75})
\begin{equation}\label{f-160407}
R_{\ast}U=-i\sum_{k=1}^{n-1}R(e_k, \overline e_k)JU,~~ U\in\Gamma(HX).
\end{equation}
By duality, we can extend the Ricci operator $R_\ast$ to $\Omega^{0, q}(X)$ in the following way
\begin{equation}\label{g-160407}
R_\ast u(\overline U_1, \ldots, \overline U_{q})=\sum_{j=1}^qu(\overline U_1, \ldots, R_\ast \overline U_j, \ldots, \overline U_q)
\end{equation}
for all $u\in\Omega^{0, q}(X)$ and $U_1, \ldots, U_q\in \Gamma(T^{1, 0}X).$
It is straightforward to check that $R_\ast$ is a self-adjoint operator with respect to the
inner product $\langle\cdot\,|\,\cdot\rangle$ on $\Omega^{0, q}(X).$
\subsection{Pseudohermitian geometry}Let $\{Z_\alpha\}_{\alpha=1}^{n-1}$ be the canonical frame of $T^{1, 0}X$ on a canonical open set $D$ in the BRT trivialization given in Theorem \ref{e-can-160301}. Then $\{dz_\alpha\}_{\alpha=1}^{n-1}$ is a dual frame of $\{Z_\alpha\}_{\alpha=1}^{n-1}$. Write $Z_{\overline\alpha}=\overline{Z_\alpha}$, $\theta^\alpha=dz_{\alpha}$, $\theta^{\overline\alpha}=\overline{\theta^\alpha}.$ Then $\{\theta^\alpha\}$ is an admissible coframe on $D$. Then $d\omega_0=ig_{\alpha\overline\beta}\theta^\alpha\wedge\theta^{\overline\beta}$, where $g_{\alpha\overline\beta}=2\frac{\partial^2\varphi(z)}{\partial z_\alpha\partial\overline z_\beta}.$ Let $\omega^{\beta}_{\alpha}$ be the connection form of Tanaka-Webster connection with respect to the frame $\{Z_\alpha\}_{\alpha=1}^{n-1}.$ Thus, we have $$\nabla Z_\alpha=\omega_{\alpha}^\beta Z_\beta, \nabla Z_{\overline\alpha}=\omega_{\overline\alpha}^{\overline\beta} Z_{\overline\beta}, \nabla T=0.$$ By direct calculation, \begin{equation} \omega_{\alpha}^\beta=g^{\overline\sigma\beta}\partial g_{\alpha\overline\sigma} \end{equation}where $\{g^{\overline\sigma\beta}\}$ is the inverse matrix of $\{g_{\alpha\overline\beta}\}$. We denote by $\Theta_{\alpha}^\beta$ the Tanaka-Webster curvature form. Since the pseudohermitian torsion vanishes, we have $\Theta_{\alpha}^\beta=d\omega_{\alpha}^\beta-\omega_{\alpha}^\gamma\wedge\omega_\gamma^\beta.$ It is easy to check that
$\Theta_{\alpha}^\beta=R_{\alpha\ j\overline k}^{\ \ \beta}\theta^j\wedge\theta^{\overline k}$, where $R_{\alpha\ j\overline k}^{\ \ \beta}$ is the Tanaka-Webster curvature and by direct calculation
\begin{equation}
R_{\alpha\ j\overline k}^{\ \ \beta}\theta^j\wedge\theta^{\overline k}=-2g^{\overline\sigma\beta}\frac{\partial^4\varphi(z)}{\partial z_\alpha\partial\overline z_\sigma\partial z_j\partial\overline z_k}dz_j\wedge d\overline z_k-2\frac{\partial g_{\alpha\overline\sigma}}{\partial z_j}\frac{\partial g^{\overline\sigma\beta}}{\partial \overline z_k}dz_j\wedge d\overline z_k.
\end{equation}

\begin{proposition}[{\cite[Theorem 5.2]{T75}}]\label{c-160228}
For any $u\in\Omega^{0, q}(X)$, we have the following equalities
\begin{equation}\label{e-160228}
(\Box^{(q)}_b u| u)=\|u\|_{\overline S}^2-qi(\nabla_{T} u|u)
+(R_\ast u|u)
\end{equation}
where $\|u\|_{\overline S}^2=
-\int_X(\sum_{k=1}^{n-1}\langle\nabla_{e_k}\nabla_{\overline e_k} u|u\rangle)dv_X$.
Here, $\{e_k\}_{k=1}^{n-1}$ is any orthonormal frame of $T^{1, 0}X$.
\end{proposition}
From (2) in Proposition \ref{p-160327}, the rigid Hermitian metric $\langle\cdot|\cdot\rangle_g$ is invariant with respect the Tanaka-Webster connection $\nabla$. Integrating by parts, we have \begin{equation}\label{p-161113}
\|u\|_{\overline S}^2=
-\int_X\sum_{k=1}^{n-1}\langle\nabla_{e_k}\nabla_{\overline e_k} u|u\rangle dv_X=\sum_{k=1}^{n-1}\int_X\langle\nabla_{\overline e_k} u|\nabla_{\overline e_k}u\rangle dv_X\geq0.
\end{equation}
As a corollary of Proposition \ref{c-160228}, we have the vanishing theorem
for the Fourier components of Kohn-Rossi  cohomology.
\begin{theorem}
Let $X$ be a strictly pseudoconvex CR manifold with a locally free transversal CR $S^1$-action.
There exists $m_0>0$ such that for $q\geq1$ and any $m\in\mathbb Z$
with $m>m_0$, we have $H^q_{b, m}(X)=0.$
\end{theorem}
\begin{proof}
By Lemma \ref{d-160228}, we have $\nabla_T u=L_T u$.
Then by (\ref{e-160228}) for any $u\in\Omega^{0, q}_m(X)$ we have
\begin{equation}\label{e-160308}
(\Box^{(q)}_{b, m}u| u)=\|u\|_{\overline S}^2+qm\|u\|^2+(R_\ast u| u).
\end{equation}
There exists $m_0>0$ such that for any $m>m_0, m\in\mathbb N$ we have
\begin{equation}
(\Box^{(q)}_{b, m}u|u)\geq C_m\|u\|^2 ~\text{for}~u\in\Omega^{0, q}_m(X), q\geq1.
\end{equation}
This implies  $\mathcal H^q_{b, m}(X)=0$ for $m>m_0$, $q\geq 1.$ By the Hodge isomorphism \eqref{i},
we get the conclusion of the theorem.
\end{proof}
\section{$S^1$-invariant deformation of the
  CR structure}\label{section-160322}
Let $\{J_t\}_{t\in(-\delta,\delta)}$ be  a deformation of $J$.
As before, let $T_t^{1, 0}X=\{U\in \mathbb CHX: J_t U=iU\}$.
We also say  $T_t^{1, 0}X$  is a (smooth) deformation of $T^{1, 0}X$.
In this work, we are especially interested in the
$S^1$-invariant deformations of CR structures. As in Definition~\ref{g-160301b}, we introduce:
\begin{definition}\label{g-160301}
We say the smooth deformation $T_t^{1, 0}X$ of  $T^{1, 0}X$ is
$S^1$-invariant if for any $t\in(-\delta, \delta)$ we have $L_TJ_t=0$ or, equivalently,
$[T, \Gamma(T_t^{1, 0}X)]\subset\Gamma(T_t^{1, 0}X)$.
\end{definition}
We give some examples of $S^1$-invariant deformations of CR structures.
\begin{example}\label{example-160328}
Let $X=\mathbb S^3=\{z=(z_1, z_2)\in\mathbb C^2: |z_1|^2+|z_2|^2=1\}$ be the boundary of the unit ball in $\mathbb C^2$, and let the induced CR structure $T^{1, 0}X$ be generated by $Z=\overline z_2\frac{\partial }{\partial z_1}-\overline z_1\frac{\partial}{\partial z_2}$. Thus, $(X, T^{1, 0}X)$ forms a compact strictly pseudoconvex CR manifold. The $S^1$-action on $X$ is given  by
\begin{equation}\label{action-160330}
e^{i\theta}(z_1, z_2)=(e^{i\theta}z_1, e^{in\theta}z_2), n\in\mathbb Z, n>0.
\end{equation}
By direct calculation, the $S^1$- action given above is a locally free transversal CR $S^1$- action. The global vector field induced by the $S^1$-action on $X$ is given by
$$
T=i\Big(z_1\frac{\partial}{\partial z_1}-
\overline z_1\frac{\partial}{\partial \overline z_1}+
nz_2\frac{\partial}{\partial z_2}
-n\overline z_2\frac{\partial }{\partial \overline z_2}\Big).
$$
By  simple calculation,
\begin{equation}
[T, Z]=-i(n+1)Z.
\end{equation}
Let $\Phi(z, t)=\phi(z)\chi(t)$ with $\phi(z)$ and $\chi(t)$ smooth functions on $X$ and $\mathbb R$ respectively. We assume that $T\phi(z)=-2i(n+1)\phi$ with $\phi$ a non-zero smooth function on $X$ and $\chi(0)=0.$  This is possible because we can find a smooth function $h$ on $X$ such that $\int_0^{2\pi } h(e^{i\theta}z)e^{i2(n+1)\theta}d\theta\neq 0$, then we define $\phi(z)=\int_0^{2\pi } h(e^{i\theta}z)e^{i2(n+1)\theta}d\theta$. Then for each $t\in\mathbb R$ the  deformation $T_t^{1, 0}X$ of  $T^{1, 0}X$ is given by
\begin{equation}
T_t^{1, 0}X={\rm span}_{\mathbb C}\{Z+\Phi(z, t)\overline Z\}.
\end{equation}
It is easy to check that
\begin{equation}
[T, Z+\Phi(z, t)\overline Z]=-i(n+1)(Z+\Phi(z, t)\overline Z)\in\Gamma(T_t^{1, 0}X).
\end{equation}
Thus, $T_t^{1, 0}X$ is an $S^1$-invariant deformation of ~$T^{1, 0}X.$
\end{example}
\begin{remark}
In Rossi's  global non-embeddability example
\cite{AS70,Bur:79,G94,Ro65}, a real analytic deformation of
$T^{1, 0}\mathbb S^3$ was considered.
For each $t\in\mathbb R$, the new CR structure $T_t^{1, 0}\mathbb S^3$ on
$\mathbb S^3$ is generated by $Z+t\overline Z$.
It is easy to check that this  is not a $S^1$-invariant deformation with
respect to the $S^1$-action given in (\ref{action-160330}).
\end{remark}
Now, we assume that $\{J_t\}_{t\in(-\delta,\delta)}$ is a $S^1$-invariant deformation of $J$.
Then, $\mathbb CHX=T^{1, 0}_tX\bigoplus T^{0, 1}_tX$, that is,
the deformations are always horizontal.
This implies that the $S^1$-action on $X$ is transversal.
From Definition \ref{g-160301}, we know that the $S^1$-action on $X$
is a transversal CR $S^1$-action with respect to the deformation
$T_t^{1, 0}X$ for $t\in(-\delta, \delta)$.

We now express  $T_t^{1, 0}X$ in an explicit way. Let $\{Z_j\}_{j=1}^{n-1}$ be a canonical frame of $T^{1, 0}X$ defined in Theorem \ref{e-can-160301}. Then locally we have
\begin{equation}\label{e-gue160411}
T_t^{1, 0}X={\rm Span}_{\mathbb C}\left\{Z_j+\sum_{k=1}^{n-1}\Phi_{j\overline k}(\cdot, t)\overline Z_k, j=1, \ldots, n-1\right\}
\end{equation}
for $|t|$ small. We may assume that \eqref{e-gue160411} holds for all $t\in(-\delta,\delta)$.
$ Z_j+\sum_{k=1}^{n-1}\Phi_{j\overline k}(\cdot, t)\overline Z_k$, $ 1\leq j\leq n-1$ is a basis of CR structure $T_t^{1, 0}X$. Here, $\{\Phi_{j, k}(\cdot, t)\}_{1\leq j, k\leq n-1}$ is called deformation matrix
and $\{\Phi_{j\overline k}\}_{j, k=1}^{n-1}$ are smooth functions on $X$ which smoothly depend on $t\in(-\delta, \delta).$
\begin{lemma}
With the notations used above,
for any $t\in (-\delta, \delta)$, we have
$T\Phi_{j\overline k}=0$ for $1\leq j, k\leq n-1.$
\end{lemma}
\begin{proof}
Since the $S^1$-action on $X$ is transversal and CR with respect to the CR structure $J_t$, we have $[T, Z_j+\sum_{ k=1}^{n-1}\Phi_{j\overline k}\overline Z_k]\in\Gamma(T^{1, 0}_tX)$. From Theorem \ref{e-can-160301}, we know $[T, Z_j]=[T, \overline Z_j]=0$ for $1\leq j\leq n-1.$ Then $[T, Z_j+\sum_{k=1}^{n-1}\Phi_{j\overline k}\overline Z_k]=\sum_{k=1}^{n-1} T\Phi_{j\overline k}\overline Z_k\in\Gamma(T_t^{1, 0}X).$ At each point, we write $\sum_{k=1}^{n-1} T\Phi_{j\overline k}\overline Z_k=\sum_{j=1}^{n-1}c_j(Z_j+\sum_{l=1}^{n-1}\Phi_{j\overline l}\overline Z_l) $ for constants $c_l, 1\leq l\leq n-1.$  The equality implies that $c_l=0, 1\leq l\leq n-1$, i.e., $\sum_{k=1}^{n-1} T\Phi_{j\overline k}\overline Z_k=0$. Since $\{\overline Z_l\}_{l=1}^{n-1}$ are linear independent, we have that $T\Phi_{j\overline k}=0$ for $1\leq j, k\leq n-1.$
\end{proof}
Associated with the CR structure tensor $J_t, t\in(-\delta, \delta)$, the Levi form on $X$ is defined by
\begin{equation}
\mathcal L_{t, x}(U, V)=-d\omega_0(J_tU, V), \forall~ U, V\in H_x X, \forall~x\in X.
\end{equation}
When $\delta$ is sufficiently small, the quadratic form $\mathcal L_{t, x}$ is still positive and we may assume that the CR manifold $(X, T^{1, 0}_tX)$ is strictly pseudoconvex for $t\in(-\delta, \delta).$ Since the $S^1$-action on $(X, T_t^{1, 0}X)$ is  transversal and CR, using $\mathcal L_{t, x}$ we can define a Riemannian metric $g_t$ on $\mathbb CTX$ as  (\ref{a-160227}).
As \eqref{e-gue160410}, $g_t$ induces a rigid Hermitian metric $\langle\cdot\,|\,\cdot\rangle_t$ on $\mathbb CTX$ such that
\begin{equation}
T^{1, 0}_tX\perp T^{0, 1}_tX,\quad  T\perp T^{1, 0}_tX\oplus T^{0, 1}_tX.
\end{equation}

Denote by $T_t^{\ast 1,0}X$ and $T_t^{\ast0,1}X$ the dual bundles of
$T_t^{1,0}X$ and $T_t^{0,1}X$ respectively and define the vector bundle of $(0,q)$-forms by
$\Lambda^qT_t^{\ast0,1}X$.
Similarly as in Section \ref{prem160323}, let $\Omega^{0, q}_t(X)$ denote
the space of global smooth sections of $\Lambda^qT_t^{\ast0, 1}X$ and for every
$m\in\mathbb Z$, let $\Omega^{0, q}_{t, m}(X)=\{u\in\Omega^{0, q}_t(X): Tu=imu\}.$
Let $\overline\partial_{t, b}: \Omega^{0, q}_t(X)\rightarrow\Omega^{0, q+1}_t(X)$
be the tangential Cauchy-Riemann operator with respect to the new CR structure
$T_t^{1, 0}X$. Then we still have that $T\overline\partial_{t, b}=\overline\partial_{t, b}T$
and $\overline\partial_{t, b, m}:=
\overline\partial_{t, b}: \Omega^{0, q}_{t,m}(X)\rightarrow
\Omega^{0, q+1}_{t,m}(X)$, for every $m\in\mathbb Z$.
Using the $\overline\partial_{t, b}$-complex,
$\overline\partial_{t, b, m}$-complex on
$\Omega^{0, q}_{t}(X)$, $\Omega^{0, q}_{t, m}(X)$
respectively, we can define the Kohn-Rossi cohomology
$H^q_{t, b}(X)$ and the $m$-th Fourier component of
Kohn-Rossi cohomology $H^q_{t, b, m}(X)$ for each
$m\in\mathbb Z$ respectively, $q=0,1,\ldots,n-1$.
In the remainder of this section, our goal is to prove the following:
\begin{theorem}\label{main-160302}
Let $(X, HX, J)$ be a compact strictly pseudoconvex CR
manifold of real dimension $2n-1, n\geq 2$ with a locally free transversal
CR $S^1$-action. Let $\{J_t\}_{t\in(-\delta, \delta)}$ be a
$S^1$-invariant deformation of $J$.
Then there exists positive constants $m_0$ and
$\delta_0<\delta$ such that for $m\in\mathbb Z$, $m>m_0$ and $|t|<\delta_0$,
\begin{equation}\label{e-160314-1}
(\Box^{(q)}_{t, b, m}u|u)_t\geq
C_m\|u\|_t^2, \:\:u\in\Omega^{0, q}_{t, m}(X),\:\: q\geq1,
\end{equation}
where $C_m$ is a constant independent of $t$. In particular, we have the simultaneous vanishing
\begin{equation}
H^q_{t, b, m}(X)=0, \:\:m> m_0,\:\: |t|<\delta_0,\:\: q\geq1.
\end{equation}
\end{theorem}
Before the proof of  Theorem \ref{main-160302}, we first recall the harmonic theory with respect to  $\{J_t\}_{t\in(-\delta, \delta)}$ on $X$.
Let $(\,\cdot\,|\,\cdot\,)_t$ be the $L^2$ inner product on $\Omega^{0,q}_t(X)$ induced by the rigid Hermitian metric $\langle\cdot\,|\,\cdot\rangle_t$ and let $\|\cdot\|_t$ denote the corresponding norm. Then for all $u, v\in\Omega^{0,q}_t(X)$
\begin{equation}\label{h-160302}
(u|v)_t=\int_X\langle u| v\rangle_t dv_X,
\end{equation}
where $dv_X$ is the volume form on $X$ induced by the rigid Hermitian metric
$\langle\cdot\,|\,\cdot\rangle_t$. Recall that the volume
$dv_X=\omega_0\wedge\frac{(d\omega_0)^{n-1}}{(n-1)!}$
associated with $\langle\cdot\,|\,\cdot\rangle_t$ does not depend on $t$.
Let
$\overline\partial_{t, b}^\ast:\Omega^{0, q}_{t}(X)\rightarrow\Omega^{0, q-1}_{t}(X)$
be the formal adjoint of $\overline\partial_{t, b}$ with respect to
$(\cdot\,|\,\cdot)_t$ for $t\in(-\delta, \delta)$.
Since $\overline\partial_{t, b}T=T\overline\partial_{t, b}$ and the Hermitian metric
$\langle\cdot\,|\,\cdot\rangle_t$ is rigid, we have
$T\overline\partial_{t, b}^\ast=\overline\partial_{t, b}^\ast T$.
Define $\Box^{(q)}_{t, b}=\overline\partial_{t, b}\overline\partial_{t, b}^\ast+
\overline\partial_{t, b}^\ast\overline\partial_{t, b}$.
From the commutation of $T$ with $\overline\partial_{t, b}$, $\overline\partial_{t, b}^\ast$,
we have $\Box^{(q)}_{t, b}T=T\Box^{(q)}_{t, b}.$
Then $\Box^{(q)}_{t, b}$ maps $\Omega^{0, q}_{t, m}(X)$
into itself and we denote
\[
\Box^{(q)}_{t, b, m}:=\Box^{(q)}_{t, b}\Big|_{\Omega^{0, q}_{t, m}(X)}:
\Omega^{0, q}_{t, m}(X)\rightarrow\Omega^{0, q}_{t, m}(X),
\]
the restriction of $\Box^{(q)}_{t, b}$ to $\Omega^{0, q}_{t, m}(X)$.
As in Section \ref{prem160323}, let $L^2_{t, (0, q), m}(X)$ be the completion
of $\Omega^{0, q}_{t, m}(X)$ under the $L^2$ inner product defined in \eqref{h-160302}.
We extend $\Box^{(q)}_{t, b, m}$ to $L^2_{t, (0, q), m}(X)$ as in \eqref{j2}.
By Hodge theory for $\Box^{(q)}_{t, b, m}$ (Theorem \ref{gI})
there is an isomorphism $H^q_{t, b, m}(X)\cong \mathcal H^q_{t, b, m}(X)$,
where $\mathcal H^q_{t, b, m}(X)$ is the kernel of $\Box^{(q)}_{t, b, m}$.
Now we are going to show the simultaneous vanishing theorem for
the harmonic space $\mathcal H^q_{t, b, m}(X), q\geq 1$  and as a consequence we prove Theorem \ref{main-160302}.
\begin{proof}[Proof of Theorem \ref{main-160302}]
Since  $\{Z_{t, j}=Z_j+\Phi_{j\overline k}(\cdot, t)Z_{\overline k}\}_{j=1}^{n-1}$
is a frame of $T_t^{1, 0}X$  and  $\langle\cdot\,|\,\cdot\rangle_t$ depends smoothly on $t$ ,
then by linear algebra argument we can find an orthonormal frame of
$T_t^{1, 0}X$ which depends smoothly on $t$.
Locally, let $\{e_{t, j}\}_{j=1}^{n-1}$ be an orthonormal basis
of $T_t^{ 1, 0}X$ with respect to  $\langle\cdot\,|\,\cdot\rangle_t$
depending smoothly on $t$ and let $\{\omega_{t}^j\}_{j=1}^{n-1}$
be its dual basis.

Let $\nabla^{t}$ be the Tanaka-Webster connection with respect to $T_t^{1, 0}X$ and $\langle\cdot\,|\,\cdot\rangle_t$ for any $t\in(-\delta, \delta).$
Let $R^t$ and $R^t_\ast$ be its curvature and Ricci curvature operator respectively defined as in (\ref{f-160407}) and (\ref{g-160407}).  For any $u\in\Omega^{0, q}_t(X)$, then locally $u=\sum_{|J|=q}^\prime u_J\overline \omega^J_t$, where $\sum^\prime$ means that the summation is performed only over strictly increasing multi-indices. Here for a multi-index  $J=\{j_1, \ldots, j_q\}\in\{1,2,\ldots,n-1\}^q$, we set $\abs{J}=q$, $\overline \omega_t^J=\overline\omega_t^{j_1}\wedge\ldots\wedge\overline\omega_t^{j_q}$ and we say that $J$ is strictly increasing if $1\leq j_1<\ldots<j_q\leq n-1$. By definition of $R_\ast^t$, for any strictly increasing multi-index $1\leq k_1<\ldots<k_q\leq n-1$
\begin{equation}
R_\ast^tu(\overline e_{t, k_1}, \ldots, \overline e_{t, k_q})=\sum_{j=1}^q u(\overline e_{t, k_1}, \ldots, R^t_\ast \overline e_{t, k_j}, \ldots, \overline e_{t, k_q}),
\end{equation}
where \begin{equation}
R_{\ast}^t\overline e_{t, k_j}=-\sum_{i=1}^{n-1}R^t(e_{t, i}, \overline e_{t, i})\overline e_{t, k_j}.
\end{equation}
By (\ref{e-160308}), for any $u\in\Omega^{0, q}_{t, m}(X)$ we have
\begin{equation}\label{g-160308}
(\Box^{(q)}_{t, b, m}u|u)_t=\|u\|_{\overline S_t}^2+qm\|u\|_t^2+(R^t_\ast u|u)_t
\end{equation}
where $\|u\|_{\overline S_t}^2=-\sum_{i=1}^{n-1}\int\langle\nabla^t_{e_{t, i}}\nabla^t_{\overline e_{t, i}}u|u\rangle_tdv_X$. We claim that
$
(R^t_\ast u| u)_t\leq C\|u\|_t^2, \forall ~|t|\leq \delta
$
for a constant $C$ independent of $t$ when $\delta$ is small. For $u=\sum^\prime_{|J|=q}u_J\overline\omega_t^J\in\Omega^{0,q}_t(X)$, write $R^t_\ast u=\sum_{|J|=q}^\prime u^t_J\overline\omega^J_t.$ For any $J=\{j_1, \ldots, j_q\}$ with $1\leq j_1<\ldots<j_q\leq n-1$ we have $u^t_J=R^t_\ast u(\overline e_{t, j_1}, \ldots, \overline e_{t, j_q}).$ Then
\begin{equation}
\begin{split}
(R^t_\ast u| u)_t&=\sum_{ j_1<\ldots<j_q}\int_X R_\ast^tu(\overline e_{t, j_1}, \ldots, \overline e_{t, j_q})\overline {u_{j_1\ldots j_q}} dv_X\\
&=\sum_{ j_1<\ldots<j_q}\sum_{l=1}^q\int_X u(\overline e_{t, j_1}, \ldots, R_\ast^t \overline e_{t, j_l}, \ldots, \overline e_{t, j_q})\overline{u_{j_1\ldots j_q}}dv_X\\
&=-\sum_{i=1}^{n-1}\sum_{ j_1<\ldots<j_q}\sum_{l=1}^q\int_Xu(\overline e_{t, j_1}, \ldots, R^t(e_{t, i}, \overline e_{t, i})\overline e_{t, j_l}, \ldots, \overline e_{t, j_q})\overline{u_{j_1\ldots j_q}}dv_X.
\end{split}
\end{equation}
Since $J_t$,  $\langle\cdot\,|\,\cdot\rangle_t$ and  $\{Z_{t, j}\}$ depend smoothly on $t$, then the connection forms of $\nabla^t$ with respect to the frame $\{Z_{t, j}\}$ depend smoothly on $t$ and as a consequence, the curvature  of the  $\nabla^t$ also depend smoothly on $t$ . Thus, there exists a constant $\delta_0$ such that for any $|t|<\delta_0$ we have
\begin{equation}\label{h-160308}
(R^t_\ast u| u)_t\leq C\|u\|_t^2
\end{equation}
for some constant $C$ independent of $t$. From (\ref{g-160308}) and (\ref{h-160308}), there exist a constant $m_0>0$ independent of $t$ such that for any $m\in\mathbb Z, m>m_0$ we have
\begin{equation}\label{e-160314}
(\Box^{(q)}_{t, b, m}u|u)_t\geq C_m\|u\|_t^2,~\forall~ u\in\Omega^{0, q}_{t, m}(X), |t|<\delta_0, q\geq1,
\end{equation}
where $C_m$ is a constant independent of $t$ for $|t|<\delta_0$. From (\ref{e-160314}), we get $\mathcal H^q_{t, b, m}(X)=0$ for any $m\in\mathbb Z, m>m_0$ and $|t|<\delta_0.$ By Hodge theory we get the conclusion of Theorem \ref{main-160302}.
\end{proof}
From (\ref{e-160314}) we have the following.
\begin{corollary}\label{c-gue160325}
Let $\lambda(t,m)$ be an eigenvalue of $\Box^{(q)}_{t, b, m}$\,, $1\leq q\leq n-1$.
Assume that $m_0, \delta_0$ are the same as in Theorem \ref{main-160302}.
Then, for any $m\in\mathbb Z, m>m_0$ and $\abs{t}<\delta_0$,
we have $\lambda(t,m)\geq C_m$. Here, $C_m$ is a constant satisfying
$C_1m\leq C_m\leq C_2 m$ with $C_1$, $C_2$ independent of
$m$ and $t$, $|t|<\delta_0$.
\end{corollary}
\section{Stability of  Szeg\H{o} kernel of the Fourier components
of CR functions}\label{stability-160323}

In this section, we assume $m_0$, $\delta_0$ be the same constants
as in Theorem \ref{main-160302} unless otherwise stated.
Let $S_{t, m}:L^2(X)\To\mathcal H^0_{t, b, m}(X)$ be the
orthogonal projection with respect to $(\,\cdot\,|\,\cdot\,)_t$.
Since the volume form $dv_X$ with respect to $\langle\cdot\,|\,\cdot\rangle_t$
does not depend on $t$, the inner product $(\,\cdot\,|\, \cdot\,)_t$
is the same as $(\, \cdot\,|\,\cdot\,)$ on the space of smooth functions on $X$.
Let $S_{t,m}(x,y)\in C^\infty(X\times X)$ be the Schwartz kernel of $S_{t,m}$.
We denote $S_m:=S_{0,m}$, $S_m(x,y):=S_{0,m}(x,y)$. The goal of this section is to prove the following
\begin{theorem}\label{main2-160308}
With the notations above, assume that $m\geq m_0$.
For any $k\in\mathbb N$ and $\varepsilon>0$ there exists
$\delta_{k, \varepsilon}<\delta_0$ such that for all
$t\in\mathbb R$ with $|t|<\delta_{k, \varepsilon}$, we have
\begin{equation}\label{e-gue160325}
|S_{t, m}(x, y)-S_m(x, y)|_{C^k(X\times X)}<\varepsilon.
\end{equation}
\end{theorem}

For $s\in\mathbb Z$, let $H^{s}(X)$ denote the Sobolev space on $X$
of order $s$  of functions and let $\norm{\cdot}_{s}$
denote the standard Sobolev norm of order $s$ with respect to $(\,\cdot\,|\,\cdot\,)$. First, we need

\begin{lemma}\label{l-gue160325}
For every $m\geq m_0$ and every $s_0\in\mathbb N\cup\{0\}$,
there is a constant $C_{s_0,m}>0$ independent of $t\in(-\delta_0, \delta_0)$ such that
\begin{equation}\label{e-gue160325I}
\norm{S_{t, m}u}_{2s_0}\leq C_{s_0,m}\norm{u}_{-2s_0}\ \
~\text{for}~ u\in H^{-2s_0}(X) ~\text{and}~ |t|<\delta_0.
\end{equation}
\end{lemma}

\begin{proof}
Fix  $m\geq m_0$. By G{\aa}rding's inequality, for every $s\in\mathbb N_0$, it is easy to see that there is a constant $C_{s,m}>0$ independent of $t\in(-\delta_0, \delta_0)$ such that
\begin{equation}\label{e-gue151031r}
\|S_{t,m}u\|_{s+2}\leq C_{s,m}\Bigr(\|(\Box^{(0)}_{t,b,m}-T^2)S_{t,m}u\|_{s}+
\|S_{t,m}u\|_{s}\Bigr),\ \ \forall ~u\in L^2(X).
\end{equation}
From \eqref{e-gue151031r}, by using induction and noticing that
\[\begin{split}
&T^2S_{t,m}u=-m^2S_{t,m}u,\ \ \forall u\in L^2(X),\\
&\norm{S_{t,m}u}\leq\norm{u},\ \ \forall u\in L^2(X),
\end{split}\]
it is straightforward to see that for every $s\in\mathbb N$, there is a $\Td C_{s,m}>0$ independent of $t$ such that
\begin{equation}\label{e-gue151031rI}
\norm{S_{t,m}u}_{2s}\leq \Td C_{s,m}\norm{u},\ \ \forall u\in L^2(X).
\end{equation}
From (\ref{e-gue151031rI}), it is straightforward to see that $S_{t, m}$ can be extended from $L^2(X)$ to $H^{-2s}(X)$ for every $s\in\mathbb N.$
Fix $s_0\in\mathbb N$ and let $u\in H^{-2s_0}(X)$, we have $(S_{t, m}u|v)=(u, S_{t, m}v)$ for $v\in L^2(X),$ where  $(\cdot, \cdot)$ is the pair between $H^{-2s}(X)$ and $H^{2s}(X)$. Then
\begin{equation}\label{e-gue160325a}
\norm{S_{t,m}u}=\sup\set{\abs{(\,u\,, \,S_{t,m}v\,)}:\, v\in L^2(X),\ \ (\,v\,|\,v\,)=1}.
\end{equation}
Fix $v\in L^2(X)$, $(\,v\,|\,v\,)=1$. From \eqref{e-gue151031rI}, we have
\begin{equation}\label{e-gue160325aI}
\begin{split}
\abs{(\,u\,,\,S_{t,m}v\,)}\leq\norm{u}_{-2s_0}\cdot\norm{S_{t,m}v}_{2s_0}\leq\Td C_{s_0,m}\norm{u}_{-2s_0},
\end{split}
\end{equation}
where $\Td C_{s_0,m}>0$ is the constant as in \eqref{e-gue151031rI}. From \eqref{e-gue160325aI} and \eqref{e-gue160325a}, we conclude that
\begin{equation}\label{e-gue160325aII}
\norm{S_{t,m}u}\leq \Td C_{s_0,m}\norm{u}_{-2s_0},\ \ \forall u\in H^{-2s_0}(X).
\end{equation}
Now, from \eqref{e-gue151031rI} and \eqref{e-gue160325aII}, we have
\begin{equation}\label{e-gue160325b}
\begin{split}
\norm{S_{t,m}u}_{2s_0}&=\norm{S_{t,m}S_{t,m}u}_{2s_0}\leq\Td C_{s_0,m}\norm{S_{t,m}u}\\
&\leq(\Td C_{s_0,m})^2\norm{u}_{-2s_0},\ \ \forall u\in H^{-2s_0}(X).
\end{split}
\end{equation}
From \eqref{e-gue160325b}, the lemma follows.
\end{proof}

Let $N_{t,m}:L^2_m(X)\To{\rm Dom\,}(\Box^{(0)}_{t,b,m})$ be the partial inverse of $\Box^{(0)}_{t,b,m}$. We have
\begin{equation}\label{e-gue160325p}
\begin{split}
&\Box^{(0)}_{t,b,m}N_{t,m}+S_{t,m}=I\ \ \mbox{on $L^2_m(X)$},\\
&N_{t,m}\Box^{(0)}_{t,b,m}+S_{t,m}=I\ \ \mbox{on ${\rm Dom\,}(\Box^{(0)}_{t,b,m})$}.\\
\end{split}
\end{equation}
We denote $N_m:=N_{0,m}$. We need

\begin{lemma}\label{l-gue160325I}
For every $m\geq m_0$ and every $s\in\mathbb N_0$, there is a constant $C_{s,m}>0$ independent of $t\in(-\delta_0, \delta_0)$ such that
\begin{equation}\label{e-gue160325pII}
\norm{N_{t, m}u}_{s+2}\leq C_{s,m}\norm{u}_{s}~\text{for}~ u\in H^{s}(X)\bigcap L^2_m(X).
\end{equation}
\end{lemma}

\begin{proof}
We will prove \eqref{e-gue160325pII} by induction over $s\in\mathbb N_0$. By G{\aa}rding's inequality, it is easy to see that there is a constant $\Td C_{m}>0$ independent of $t$ such that
\begin{equation}\label{e-gue160325y}
\|N_{t,m}u\|_{2}\leq\Td C_{m}\Bigr(\|(\Box^{(0)}_{t,b,m}-T^2)N_{t,m}u\|+
\|N_{t,m}u\|\Bigr),\ \ \forall u\in L^2_m(X).
\end{equation}
From \eqref{e-gue160325p}, we have
\begin{equation}\label{e-gue160325yI}
(\Box^{(0)}_{t,b,m}-T^2) N_{t,m}u=(I-S_{t,m})u+m^2N_{t,m}u.
\end{equation}
From \eqref{e-gue160325yI} and Corollary~\ref{c-gue160325}, we see that there is a constant $\hat C_m>0$ independent of $t$ such that
\begin{equation}\label{e-gue160325yII}
\norm{N_{t,m}u}+\norm{(\Box^{(0)}_{t,b,m}-T^2)N_{t,m}u}\leq\hat C_m\norm{u},\  \ \forall u\in L^2_m(X).
\end{equation}
From \eqref{e-gue160325yII} and \eqref{e-gue160325y}, we see that \eqref{e-gue160325pII} holds for $s=0$.

We assume that \eqref{e-gue160325pII} holds for some $s_0\geq0$. We are going to prove that \eqref{e-gue160325pII} holds for $s_0+1$. By G{\aa}rding's inequality, it is easy to see that there is a constant $\Td C_{s_0,m}>0$ independent of $t$ such that
\begin{equation}\label{e-gue160325yIII}
\begin{split}
&\|N_{t,m}u\|_{s_0+3}\\
&\leq\Td C_{s_0,m}\Bigr(\|(\Box^{(0)}_{t,b,m}-T^2)N_{t,m}u\|_{s_0+1}+
\|N_{t,m}u\|_{s_0+1}\Bigr),\ \ \forall u\in H^{s_0+1}(X)\bigcap L^2_m(X).
\end{split}
\end{equation}
From \eqref{e-gue160325p}, we have
\begin{equation}\label{e-gue160325yIV}
(\Box^{(0)}_{t,b,m}-T^2)N_{t,m}u=(I-S_{t,m})u+m^2N_{t,m}u.
\end{equation}
From the proof of Lemma~\ref{l-gue160325}, we have
\begin{equation}\label{e-gue160325yV}
\norm{S_{t,m}u}_{s_0+1}\leq\norm{S_{t,m}u}_{2(s_0+1)}\leq c_{m,s_0}\norm{u}\leq c_{m,s_0}\norm{u}_{s_0+1},
\end{equation}
where  $c_{m,s_0}>0$ is a constant independent of $t$. By the induction, we have
\begin{equation}\label{e-gue160325yVI}
\norm{N_{t,m}u}_{s_0+1}\leq\norm{N_{t,m}u}_{s_0+2}\leq\hat c_{m,s_0}\norm{u}_{s_0}\leq\hat c_{m,s_0}\norm{u}_{s_0+1},
\end{equation}
where $\hat c_{m,s_0}>0$ is a constant independent of $t$. From \eqref{e-gue160325yVI}, \eqref{e-gue160325yV}, \eqref{e-gue160325yIV} and \eqref{e-gue160325yIII}, we see that \eqref{e-gue160325pII} holds for $s_0+1$. The lemma follows.
\end{proof}

Let $s_1, s_2\in\mathbb Z$. For a $t$-dependent operator $A_t:H^{s_1}(X)\To H^{s_2}(X)$, we write
\[A_t=o(t):H^{s_1}(X)\To H^{s_2}(X)\,,\quad t\to0,\]
if for every $\varepsilon>0$, there is a $\delta_1>0$ such that for all $|t|<\delta_1$, we have
\[\norm{A_tu}_{s_2}\leq\varepsilon\norm{u}_{s_1}~\text{for all}~ u\in H^{s_1}(X).\]

\begin{proof}[Proof of Theorem~\ref{main2-160308}]
Assume that $m\geq m_0$. From \eqref{e-gue160325p}, we have
\begin{equation}\label{e-gue160325s}
\begin{split}
S_m&=(N_{t,m}\Box^{(0)}_{t,b,m}+S_{t,m})S_m\\
&=N_{t,m}(\Box^{(0)}_{t,b,m}-\Box^{(0)}_{b,m})S_m+S_{t,m}S_m.
\end{split}
\end{equation}
Note that
\[(\Box^{(0)}_{t,b,m}-\Box^{(0)}_{b,m})S_m=o(t):H^{-s}(X)\To H^{s-2}(X),\ \ \forall s\in\mathbb N.\]
From this observation, \eqref{e-gue160325pII} and \eqref{e-gue160325s}, we deduce that
\begin{equation}\label{e-gue160325sI}
S_m-S_{t,m}S_m=o(t):H^{-s}(X)\To H^s(X),\ \ \forall s\in\mathbb N.
\end{equation}
Taking adjoints in \eqref{e-gue160325sI} with respect to $(\,\cdot\,|\,\cdot\,)$, we get
\begin{equation}\label{e-gue160325sV}
S_m-S_mS^*_{t,m}=o(t):H^{-s}(X)\To H^s(X),\ \ \forall s\in\mathbb N,
\end{equation}
where $S^*_{t,m}$ is the adjoint of $S_{t,m}$ with respect to  $(\,\cdot\,|\,\cdot\,)$. It is clear that
\begin{equation*}
S_{t,m}=S^*_{t,m}:H^{-s}(X)\To H^s(X),\ \ \forall s\in\mathbb N.
\end{equation*}
From this observation and \eqref{e-gue160325sV}, we conclude that
\begin{equation}\label{e-gue160325sVI}
S_m-S_mS_{t,m}=o(t):H^{-s}(X)\To H^s(X),\ \ \forall s\in\mathbb N.
\end{equation}
Similarly, from \eqref{e-gue160325p}, we have
\begin{equation}\label{e-gue160325sII}
\begin{split}
S_{t,m}&=(N_{m}\Box^{(0)}_{b,m}+S_{m})S_{t,m}\\
&=N_{m}(\Box^{(0)}_{b,m}-\Box^{(0)}_{t,b,m})S_{t,m}+S_{m}S_{t,m}.
\end{split}
\end{equation}
From Lemma~\ref{l-gue160325}, it is easy to check that
\begin{equation}\label{e-gue160325sIII}
N_{m}(\Box^{(0)}_{b,m}-\Box^{(0)}_{t,b,m})S_{t,m}=o(t):H^{-s}(X)\To H^s(X),\ \ \forall s\in\mathbb N.
\end{equation}
From \eqref{e-gue160325sIII} and \eqref{e-gue160325sII}, we deduce that
\begin{equation}\label{e-gue160325sIV}
S_{t,m}-S_{m}S_{t,m}=o(t):H^{-s}(X)\To H^s(X),\ \ \forall s\in\mathbb N.
\end{equation}
From \eqref{e-gue160325sVI} and \eqref{e-gue160325sIV}, we deduce that
\begin{equation}\label{e-gue160325sVII}
S_{m}-S_{t,m}=o(t):H^{-s}(X)\To H^s(X),\ \ \forall s\in\mathbb N.
\end{equation}

From \eqref{e-gue160325sVII} and the Sobolev embedding theorem, Theorem~\ref{main2-160308} follows.
\end{proof}
\begin{corollary}
There exists $\delta_1<\delta_0$ such that for  $m>m_0$, ${\rm dim}H^0_{t, b, m}(X)$ does not depend on $t\in(-\delta_1, \delta_1)$.
\end{corollary}
\begin{proof}
It is clear that
\begin{equation}
|{\rm dim}H^0_{t, b, m}(X)-{\rm dim}H^0_{b, m}(X)|=\left|\int_XS_{t, m}(x, x)-S_m(x, x)dv_X\right|\rightarrow0
\end{equation}
as $t\rightarrow0$ by Theorem \ref{main2-160308}.
Since $\dim H^0_{t, b, m}(X)$ is an integer, for each $m>m_0$ the function
$t\mapsto\dim H^0_{t, b, m}(X)$ is constant for $|t|$ is sufficiently small.
\end{proof}
\section{Stability of equivariant embedding of CR manifolds with $S^1$-action}\label{s_stab}

In a recent work \cite[Theorem 1.2]{HHL15} we showed:
\begin{theorem}\label{t-gue151127}
Let $(X,T^{1,0}X)$ be a compact connected strictly pseudoconvex
CR manifold with a transversal CR locally free $S^1$-action.
Then for every $m\in\mathbb N$,
there exist integers $\{m_j\}_{j=1}^N$ with $m_j\geq m$, $1\leq j\leq N$, and
CR functions $\{f_j\}_{j=1}^N$ with $f_j\in H^0_{b, m_j}(X)$
such the $S^1$-equivariant CR map
$\Phi: X\rightarrow \mathbb C^N$,  $x\mapsto (f_1(x), \ldots, f_N(x))$
is an embedding.
\end{theorem}
In this section, we choose $m_0$ as in Theorem~\ref{main-160302}.
We will show that the equivariant embedding in Theorem \ref{t-gue151127}
is stable under $S^1$-invariant deformations of CR structure.
For $|t|<\delta_0$, set $\Phi_{t, j}=S_{t, m_j}\Phi_j$.
Then $\{\Phi_{t, j}\}_{j=1}^N$ are CR functions with respect to
$T^{1, 0}_tX$ (or $J_t$). With these CR functions we define a CR map
with respect to  $T_t^{1, 0}X$ as follows
\begin{equation}\label{mainmap-160308}
\Phi_t: X\rightarrow\mathbb C^N, ~~~x\mapsto (\Phi_{t, 1}(x), \ldots, \Phi_{t, N}(x)).
\end{equation}

Now, we come to the following result, which
implies the main result of the paper, Theorem \ref{mtheorem-160323}.
\begin{theorem}
Let $(X, HX, J)$ be a compact connected strictly pseudoconvex CR manifold
with a transversal CR $S^1$-action. Let $\{J_t\}_{t\in(-\delta_0, \delta_0)}$
be a $S^1$-invariant deformation of $J$.
Let $m_0$ be as in Theorem~\ref{main-160302}.
Let $\Phi=(\Phi_1, \ldots, \Phi_N): X\rightarrow\mathbb C^N$
be an equivariant CR embedding with $\Phi_j\in H^0_{b, m_j}(X)$,
$m_j>m_0$, $1\leq j\leq N$. Then $\Phi_t$ defined in (\ref{mainmap-160308})
is a CR embedding when $\abs{t}$ is sufficiently small.
Moreover, for every $k\in\mathbb N$,
$\lim_{t\To0}\norm{\Phi_t-\Phi}_{C^k(X,\mathbb C^N)}=0$.
\end{theorem}
\begin{proof}
First, we prove $\Phi_t$ is an immersion for each $|t|$ is sufficiently small. 
Since \begin{equation}\label{ee-160407}
\Phi_{t, j}-\Phi_j=S_{t, m_j}\Phi_j-S_{m_j}\Phi_j=(S_{t, m_j}-S_{m_j})\Phi_j, 1\leq j\leq N,
\end{equation}
then by 
Theorem \ref{main2-160308}, we have  $|\Phi_{t, j}-\Phi_j|_{C^1(X)}$ is sufficiently small as $|t|\rightarrow 0$ for $1\leq j\leq N.$ Since $\Phi$ is an immersion, i.e., the rank of the Jacobian of $\Phi$ is $2n-1$, then there exists a constant $\sigma<\delta_0$ such that for $|t|<\sigma$ the rank of the Jacobian of $\Phi_t$ is always $2n-1$, that is, $\Phi_t$ is an immersion when $|t|<\sigma.$ Next, we claim that $\Phi_t$ is an injective map when $t$ is sufficiently small. We prove this claim by seeking a contradiction. If it is not true, there exists $\varepsilon_n\rightarrow0$ as $n\rightarrow\infty$ and two sequences of points $\{x_n\}, \{y_n\}\subset X$, $x_n\neq y_n$, for each $n$, such that $\Phi_{\varepsilon_n}(x_n)=\Phi_{\varepsilon_n}(y_n)$, $\forall n$. Since $X$ is compact, we assume that $x_n\rightarrow p$ and $y_n\rightarrow q$. If $p\neq q$, letting $\varepsilon_n\rightarrow 0$ we will have $\Phi(p)=\Phi(q).$ This is a contradiction with $\Phi$ an injective map. Now, we assume that $p=q.$ Then
\begin{equation}\label{e-160310}
|\Phi(x_n)-\Phi(y_n)|=|\Phi(x_n)-\Phi_{\varepsilon_n}(x_n)+\Phi_{\varepsilon_n}(y_n)-\Phi(y_n)|
=|(\Phi-\Phi_{\varepsilon_n})(x_n)-(\Phi-\Phi_{\varepsilon_n})(y_n)|.
\end{equation}
By Theorem \ref{main2-160308}, we have
\begin{equation}
\norm{\Phi_{\varepsilon_n}-\Phi}_{C^1(X,\mathbb C^N)}\rightarrow0~\text{as}~\varepsilon_n\rightarrow0.
\end{equation}
Then
\begin{equation}\label{n1-160308}
|(\Phi-\Phi_{\varepsilon_n})(x_n)-(\Phi-\Phi_{\varepsilon_n})(y_n)|\leq c_{\varepsilon_n}|x_n-y_n|,\ \ \forall n,
\end{equation}
where $c_{\varepsilon_n}$ is a sequence of constants with $c_{\varepsilon_n}\rightarrow0$ as $\varepsilon_n\rightarrow0.$ On the other hand, since $\Phi$ is an embedding, by implicit function theorem, there exists a constant $c$ independent of $\{x_n\}, \{y_n\}$ such that
\begin{equation}\label{n2-160308}
|\Phi(x_n)-\Phi(y_n)|\geq c|x_n-y_n|,\ \ \mbox{for $n$ large}.
\end{equation}
From (\ref{n1-160308}) and (\ref{n2-160308}), we get  a contradiction.
Thus, we get the injectivity of $\Phi_t$ for $|t|$ sufficiently small.
The fact that $\norm{\Phi_t-\Phi}_{C^k(X,\mathbb C^N)}\rightarrow 0$
is a direct consequence of  (\ref{ee-160407}) and Theorem \ref{main2-160308}.
\end{proof}

\begin{center}
{\bf Acknowledgement}
\end{center}

The authors thank the referee for many detailed remarks that have helped improve the presentation.

\bibliographystyle{amsalpha}

\end{document}